\documentclass[11pt]{amsart}

\usepackage{amsthm}
\usepackage{graphicx,epsfig}
\usepackage{amsmath,latexsym,amssymb,verbatim}

\def\sideremark#1{\ifvmode\leavevmode\fi\vadjust{\vbox to0pt{\vss 
      \hbox to 0pt{\hskip\hsize\hskip1em           
 \vbox{\hsize2cm\tiny\raggedright\pretolerance10000
 \noindent #1\hfill}\hss}\vbox to8pt{\vfil}\vss}}}%

\begin{document}

\pagestyle{myheadings}

\def\eR{\mathbb{R}}

\title[Oscillatory Integrals and Fractal Dimension]{Oscillatory Integrals and Fractal Dimension}

\author{J.-P.\ Rolin, D.\ Vlah, V.\ \v{Z}upanovi\'{c}}

\begin{abstract}

We study geometrical representation of oscillatory integrals with an analytic phase function and a smooth amplitude with compact support. Geometrical properties of the curves defined by the oscillatory integral depend on the type of a critical point of the phase.
We give explicit formulas for the box dimension and the Minkowski content of these curves. Methods include Newton diagrams and the resolution of singularities.

\end{abstract}
\medskip

\date{}

\maketitle



\newtheorem{theorem}{Theorem}
\newtheorem*{theorem*}{Theorem} 
\newtheorem{cor}{Corollary}
\newtheorem{prop}{Proposition}
\newtheorem{lemma}{Lemma}
\theoremstyle{remark}
\newtheorem{remark}{Remark}
\newtheorem{example}{Example}

\font\csc=cmcsc10

\def\esssup{\mathop{\rm ess\,sup}}
\def\essinf{\mathop{\rm ess\,inf}}
\def\wo#1#2#3{W^{#1,#2}_0(#3)}
\def\w#1#2#3{W^{#1,#2}(#3)}
\def\wloc#1#2#3{W_{\scriptstyle loc}^{#1,#2}(#3)}
\def\osc{\mathop{\rm osc}}
\def\var{\mathop{\rm Var}}
\def\supp{\mathop{\rm supp}}
\def\Cap{{\rm Cap}}
\def\norma#1#2{\|#1\|_{#2}}

\def\C{\Gamma}

\let\text=\mbox

\catcode`\@=11
\let\ced=\c
\def\a{\alpha}
\def\b{\beta}
\def\c{\gamma}
\def\d{\delta}
\def\g{\lambda}
\def\o{\omega}
\def\q{\quad}
\def\n{\nabla}
\def\s{\sigma}
\def\div{\mathop{\rm div}}
\def\sing{{\rm Sing}\,}
\def\singg{{\rm Sing}_\ty\,}

\def\A{{\cal A}}
\def\F{{\cal F}}
\def\H{{\cal H}}
\def\W{{\bf W}}
\def\M{{\cal M}}
\def\N{{\cal N}}
\def\S{{\cal S}}

\def\ty{\infty}
\def\e{\varepsilon}
\def\f{\varphi}
\def\:{{\penalty10000\hbox{\kern1mm\rm:\kern1mm}\penalty10000}}
\def\ov#1{\overline{#1}}
\def\D{\Delta}
\def\O{\Omega}
\def\pa{\partial}

\def\st{\subset}
\def\stq{\subseteq}
\def\pd#1#2{\frac{\pa#1}{\pa#2}}
\def\sgn{{\rm sgn}\,}
\def\sp#1#2{\langle#1,#2\rangle}

\newcount\br@j
\br@j=0
\def\q{\quad}
\def\gg #1#2{\hat G_{#1}#2(x)}
\def\inty{\int_0^{\ty}}
\def\od#1#2{\frac{d#1}{d#2}}

\def\bg{\begin}
\def\eq{equation}
\def\bgeq{\bg{\eq}}
\def\endeq{\end{\eq}}
\def\bgeqnn{\bg{eqnarray*}}
\def\endeqnn{\end{eqnarray*}}
\def\bgeqn{\bg{eqnarray}}
\def\endeqn{\end{eqnarray}}

\def\bgeqq#1#2{\bgeqn\label{#1} #2\left\{\begin{array}{ll}}
\def\endeqq{\end{array}\right.\endeqn}

\def\abstract{\bgroup\leftskip=2\parindent\rightskip=2\parindent
        \noindent{\bf Abstract.\enspace}}
\def\endabstract{\par\egroup}

\def\udesno#1{\unskip\nobreak\hfil\penalty50\hskip1em\hbox{}
             \nobreak\hfil{#1\unskip\ignorespaces}
                 \parfillskip=\z@ \finalhyphendemerits=\z@\par
                 \parfillskip=0pt plus 1fil}
\catcode`\@=11

\def\cal{\mathcal}

\def\eN{\mathbb{N}}
\def\Ze{\mathbb{Z}}
\def\Qu{\mathbb{Q}}
\def\Ce{\mathbb{C}}

\def\osd{\mathrm{osd}\,}

\section{Introduction, motivation and definitions}

This paper is a starting point of a study intended to relate the standard classification of singularities of maps with the fractal dimension and the Minkowski content of curves defined by oscillatory integrals. The close link between the theory of singularities and the investigation of oscillatory integrals is well-known, and is explained in detail in \cite{arnold-vol2}. Our purpose is to connect these notions to the analysis of fractal data of curves as it is described in \cite{tricot}. In particular we consider the \emph{box counting dimension} (also called the \emph{box dimension}), and the \emph{Minkowski content}. It is worth noticing that every rectifiable curve has a box dimension equal to $1$. Hence the box dimension is a tool to distinguish nonrectifiable curves. Notice that another commonly used fractal dimension, the Hausdorff dimension, which takes the value $1$ on every non rectifiable smooth curve, cannot distinguish between them.\\

One motivation originates in previous works, in which the behavior of a (discrete or continuous) dynamical system in the neighborhood of a singular point is analyzed through the box dimension of an orbit. For example, in \cite{zuzu}, the authors consider a family of planar polynomial vector fields, called the \emph{standard model of the Hopf-Takens bifurcation}. They prove that the box dimension of any trajectory spiraling in the neighborhood of a limit cycle of multiplicity $m$ has the box dimension $2-1/m$. They also link in \cite{zuzulien}, for a planar analytic system with a weak focus singular point, the box dimension of a spiraling trajectory and the Lyapunov coefficients of the singularity.\\

If we consider now a discrete dynamical system on the real line in the neighborhood of a fixed point, the box dimension of a discrete orbit is related to the multiplicity of the generating function. This approach, together with the standard methods combining the study of discrete and continuous systems via the use of Poincar{\'e} first return map, leads to further results (see \cite{MRZ} and \cite{zuzulien}).

It is proved in \cite{majaformal} that the formal class of an analytic parabolic diffeomorphism is fully determined by the knowledge of a fractal data of a single orbit: namely, its box dimension, its Minkowski content and another number called its \emph{residual content}.\\

Based on these considerations, it seems relevant to study the singularities of a map $f:\mathbb{R}^{n}\rightarrow\mathbb{R}^{n}$ by considering the fractal data of an oscillatory integral with a phase $f$, and its geometric representation as a plane curve parametrized by its real and imaginary part. We actually observed a relation between the type of a critical point of the phase and the box dimension of the associated curve: a ``high degeneracy'' of the critical point causes a ``big accumulation'' of the curve, which is reflected by a larger box dimension. This is the exact analogue of the phenomenon observed above for the orbits or trajectories of dynamical systems. A well-known example of this situation is the oscillatory \emph{Fresnel integral}, and its geometric representation, the \emph{Cornu spiral} (also known as \emph{clothoid} or \emph{Euler spiral}). This curve plays an important role in the problem of the construction of optimal trajectories of a planar motion with a bounded derivative of the curvature; see \cite{kostov}. Its fractal data have been computed in \cite{clothoid}. It is worth noticing that the phase function of a Fresnel integral has only non-degenerate critical points.

Our results can be summarized as follows. We consider oscillatory integrals with an analytic phase function and an amplitude with compact support. We study the graph of the oscillatory integrals $I(\tau)$, as $\tau\to\infty$, and also the curves defined in a standard way, analogously as the Cornu spiral, which are defined by the parametrization given by the real and imaginary parts of the integral $I(\tau)$. We show that the box dimension and the Minkowski content of the curves reveal the leading term of the asymptotic expansion. More precisely, the \emph{oscillation index} can be read from the box dimension, while the leading coefficient can be read from the Minkowski content in the case of Minkowski nondegeneracy. Minkowski degeneracy corresponds to a nontrivial multiplicity of the oscillation index. In particular, for phase functions of two variables, we show explicitly how to connect these notions to their Newton diagram.\\

We plan to pursue the present work in various directions. One goal is the study, from our point of view, the bifurcations in parametric families of maps and their caustics. Second, we would like to know how our results behave if we take, in the oscillatory integral, an amplitude function which is not of class $C^{\infty}$ (for example, oscillatory integrals on halfspaces). Finally, we want to develop our subject in the direction of tame, but non-analytic phase functions.\\

The main results of this paper are presented in three theorems, with respect to the dimension of the space: Theorems \ref{tm_case_n1}, \ref{tm_case_n2} and \ref{tm_case_ng2}, for $n=1$, $n=2$ and $n>2$, respectively. The main difference between the first two theorems is caused by logarithmic terms which can appear in the expansion of the integral in Theorem \ref{tm_case_n2}, while in Theorem \ref{tm_case_n1}, that is not possible. In Theorem \ref{tm_case_ng2} powers of logarithmic terms can also appear.

\subsection{The box dimension}\label{subsec_box_dim}

For $A\st\eR^N$  bounded we define the \emph{$\e$-neigh\-bour\-hood} of  $A$ as:
$
A_\e:=\{y\in\eR^N\:d(y,A)<\e\}
$.
By the \emph{lower $s$-dimensional  Minkowski content} of $A$, for $s\ge0$, we mean
$$
\M_*^s(A):=\liminf_{\e\to0}\frac{|A_\e|}{\e^{N-s}},
$$
and analogously for the \emph{upper $s$-dimensional Minkowski content} $\M^{*s}(A)$.
If $\M^{*s}(A)=\M_*^{s}(A)$, we call the common value the \emph{$s$-dimensional Minkowski content of $A$}, and denote it by $\M^s(A)$.
The lower and upper box dimensions of $A$ are
$$
\underline\dim_BA:=\inf\{s\ge0\:\M_*^s(A)=0\}
$$
 and analogously
$\ov\dim_BA:=\inf\{s\ge0\:\M^{*s}(A)=0\}$.
If these two values coincide, we call it simply the box dimension of $A$, and denote it by $\dim_BA$. This will be our situation.
If $0<\M_*^d(A)\le\M^{*d}(A)<\ty$ for some $d$, then we say
 that $A$ is \emph{Minkowski nondegenerate}. In this case obviously $d=\dim_BA$.
In the case when the lower or upper $d$-dimensional Minkowski content of $A$ is equal to $0$ or $\ty$, where $d=\dim_BA$, we say that $A$ is \emph{degenerate}.
If there exists $\M^d(A)$ for some $d$ and $\M^d(A)\in(0,\ty)$, then we say that $A$ is \emph{Minkowski measurable}.
For more details on these definitions see, e.g., Falconer \cite{falc}, and \cite{zuzu}.

\subsection{Examples of the box dimension}

\begin{enumerate}

\item  A basic example of fractal sets with a nontrivial box dimension is the $a$-string defined by $A=\{k^{-a}\colon k\in\eN\}$, where $a>0$, introduced by Lapidus; see, e.g., \cite{lapidusfrank2}. Here is $\dim_BA=1/(1+a)$.

\item Furthermore,  important examples are curves from  Tricot's formulas; see \cite[p.\ 121]{tricot}.
The box dimension of a spiral in the plane defined in the polar coordinates by $r=m\,\f^{-\a}$, $\f\ge\f_1>0$, where $\f_1$, $m>0$ and $\a\in(0,1]$ are fixed, is equal to  $2/(1+\a)$.

\item Assuming that $0<\a\le\b$, the box dimension of the graph of the function $f_{\a,\b}(x)=x^{\a}\sin(x^{-\beta})$, for $x\in(0,1]$, which is called $(\a,\b)$-chirp, is equal to $2-(\a+1)/(\b+1)$; see \cite[p.\ 121]{tricot}.
\end{enumerate}

\subsection{Oscillatory integrals}

One of the main objects of interest in this paper are the oscillatory integrals
\bgeq\label{integral}
I(\tau)=\int_{\eR^n}e^{i\tau f(x)}\phi(x) dx,\qquad \tau\in\eR ,
\endeq
where $f$ is called the phase function and $\phi$ the amplitude.

Throughout this paper in all theorems we will use the following assumptions on the phase function $f$ and the amplitude $\phi$ that we call \emph{the standard assumptions}. The amplitude function $\phi:\eR^n\to\eR$,
\begin{itemize}
\item is of class $C^\infty$,
\item is a non-negative function with compact support,
\item the point $0\in\eR^n$ is contained in the interior of the support of the function $\phi$.
\end{itemize}
The phase function $f:\eR^n\to\eR$:
\begin{itemize}
\item the point $0$ is a critical point of the function $f$,
\item $f$ is a \emph{real analytic} function in the neighborhood of its critical point~$0$,
\item the point $0$ is \emph{the only} critical point of the function $f$ in the interior of the support of the function $\phi$.
\end{itemize}

The asymptotic expansion of $I(\tau)$, as $\tau\to\infty$, depends essentially on critical points of $f$. The critical point of $f$ is a point with all partial derivatives equal to zero. The nondegenerate critical point is a point were the Hessian is regular. In that case integral (\ref{integral}) is called the Fresnel integral in the reference
Arnold et all \cite{arnold-vol2}. We use theorems  from \cite{arnold-vol2} to obtain the asymptotic expansion of $I(\tau)$ as $\tau\to\infty$, in the cases if $f$ has no critical points, has the nondegenerate or the degenerate critical point. The phase function $f$ determines exponents in the asymptotic expansion, while the amplitude function determines the coefficients.
We will discuss curves defined by the oscillatory functions
\begin{eqnarray}\label{curve}
X(\tau)&=&Re\ I(\tau),\nonumber\\
Y(\tau)&=&Im\ I(\tau),
\end{eqnarray}
for $\tau$ near $\infty$, and also the reflected functions $x(t):=X(1/t)$, $y(t):=Y(1/t)$, as $t\to 0$.

\subsection{Oscillation and singular indices}
Applying \cite[Theorem 6.3]{arnold-vol2} on (\ref{integral}) we get the asymptotic expansion
\begin{equation}\label{expansion}
I(\tau)\sim e^{i\tau f(0)}\sum_{\alpha}\sum_{k=0}^{n-1}a_{k,\alpha}(\phi)\tau^{\alpha}\left(\log\tau\right)^k,\quad\textrm{as}\ \tau\to\infty .
\end{equation}
According to the same theorem, the parameter $\alpha$ is from the set consisting of a finite set of arithmetic progressions, which depend only on the phase $\phi$, and consisting of negative rational numbers. Coefficients $a_{k,\alpha}$ depend only on the amplitude $\phi$.

The \emph{index set} of an analytic phase $f$ at a critical point is defined as the set of all numbers $\alpha$ having the property: for any neighborhood of the critical point there is an amplitude with support in this neighborhood for which in the asymptotic series (\ref{expansion}) there is a number $k$ such that the coefficient $a_{k,\alpha}$ is not equal to zero. The \emph{oscillation index} $\beta$ of an analytic phase $f$ at a critical point is the maximal number in the index set.
The \emph{multiplicity of the oscillation index} $K$ of an analytic phase $f$ at a critical point is the maximal number $k$ having the property: for any neighborhood of the critical point there is an amplitude with support in this neighborhood for which in the asymptotic series (\ref{expansion}) the coefficient $a_{k,\beta}$ is not equal to zero.

The \emph{singular index} of an analytic phase $f$ in $n$ variables at a critical point is equal to $\beta+n/2$. The \emph{multiplicity of the singular index} is the multiplicity of $\beta$.

\subsection{Oscillatory and curve dimensions}
We say that $x(t)=X(1/t)$ is oscillatory near the origin if $X(\tau)$ is oscillatory near $\tau=\ty$.
We measure the rate of oscillatority of $X(\tau)$ near $\tau=\ty$ by the rate of oscillatority of $x(t)$ near $t=0$.
More precisely, the {\it oscillatory dimension} $\dim_{osc}(X)$ (near $\tau=\ty$) is defined as the box dimension
of the graph of $x(t)$ near $t=0$. Also, we investigate the associated Minkowski contents. Analogously for $y(t)$ and $Y(\tau)$.

Given the oscillatory integral $I(\tau)$ from (\ref{integral}), we define the \emph{curve dimension} of $I(\tau)$ as the box dimension of the curve defined in the complex plane by $I(\tau)$, near $\tau=\infty$. As in the oscillatory dimension, we also investigate the associated Minkowski contents.

\begin{figure}[htp]
\begin{center}
\begin{tabular}{cc}
$f_1(x)=x^2+1$ & $f_2(x)=x^3+1$\\
\includegraphics[width=0.45\textwidth]{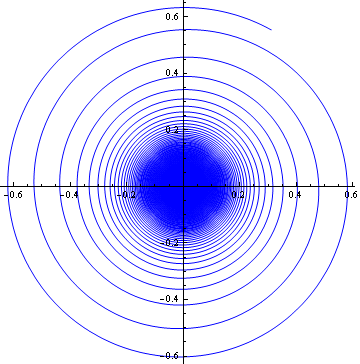}
&
\includegraphics[width=0.45\textwidth]{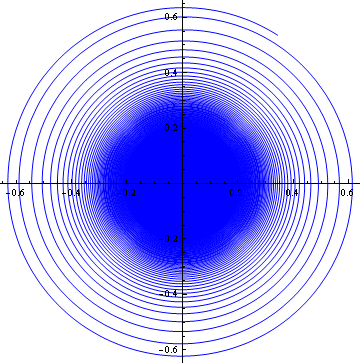}\\
$d_1=\frac{4}{3}$ & $d_2=\frac{3}{2}$

\end{tabular}
\end{center}
\caption{Curves defined by oscillatory integrals $I_i(\tau)$ from (\ref{integral}), for phase functions $f_i$ and their respective curve dimensions $d_i$, see Theorem \ref{tm_case_n1} below.}
\end{figure}

It is well known that degenerate critical points of phase functions contribute to the leading term of the asymptotic expansion (\ref{expansion}) of oscillatory integral (\ref{integral}). On the other hand, curve dimension of (\ref{curve}) will be determined by the asymptotic expansion, so we will connect type of critical point with the curve dimension. More precisely, in Theorems \ref{tm_case_n2} and \ref{tm_case_ng2}, the oscillatory and curve dimensions are related to the oscillation index. Also, it is well known that the asymptotic expansion has been related to the Newton diagram of the phase function.

\subsection{The Newton diagram}
According to \cite{arnold-vol2}, we will use the notion of the Newton polyhedron of the phase function to  formulate our results in the dimension $n\ge 2$. The Newton polyhedron is defined for the Taylor series of the critical point. Let us consider the positive orthant of the space $\eR^n$. We define the Newton polyhedron of an arbitrary subset of this orthant consisting of points with integer coordinates. At all such points we take a parallel positive orthant. The {\em Newton polyhedron} is the convex hull in $\eR^n$ of the union of all parallel orthants mentioned above. The {\em Newton diagram} $\Delta$ of a subset is the union of compact faces of the Newton polyhedron of the same subset.

We consider the power series of the phase $f$
$$f(x)=\sum a_kx^k
$$
with real coefficients, having monomials
$$
x^k=x_1^{k_1}\dots x_n^{k_n}
$$
with multi-index $k=(k_1,\dots k_n)$.
The Newton polyhedron and diagram of this power series has been constructed using the multi-indices which are in the reduced support of the series. Reduced support is obtained by removing the origin from the support of the series. This support is a subset of the positive orthant, consisting of points having non-negative coordinates. These points are given by multi-indices of all monomials from the power series, having non-zero coefficients.
The polynomial $f_\Delta$ that equals to the sum of monomials belonging to the Newton diagram, is called the {\em principal part} of the series. To each face $\gamma$ of the Newton diagram is associated the quasi-homogeneous polynomial. The type of quasi-homogeneity is determined by the slope of the face.
Furthermore, we introduce the concept of nondegeneracy of the principal part. Notice that in this  article we have $3$ distinct types of nondegeneracy:
\begin{itemize}
\item nondegeneracy of a critical point with respect to the Hessian,
\item nondegeneracy of the Minkowski content,
\item nondegeneracy of the principal part of the series.
\end{itemize}

The principal part $f_\Delta$ of the power series $f$ with real coefficients is $\eR$-nondegenerate if for every compact face $\gamma$ of the Newton polyhedron of the series the polynomials
$$
\partial f_{\gamma}/\partial x_1,\dots,\partial f_{\gamma}/\partial x_n
$$
do not have common zeroes in $(\eR\setminus 0)^n$.

Roughly speaking, $\eR$-nondegeneracy means that these mentioned derivatives have the same common zeroes as monomials. This property is essential for the resolution of the singularity, see \cite[p.\ 195]{arnold-vol2}. Furthermore, the set of all series with a degenerate principal part is `small`, more precisely, the set of $\eR$-nondegenerate series is dense in the space of all series with a fixed  Newton polyhedron, see Lemma 6.1. \cite{arnold-vol2}. A generalization of the notion of the principal part for $\eR$-degenerate vector fields could be found in \cite{zup2000}.

The asymptotic expansion of the oscillatory integrals is related to some properties of critical points of its phase function, which could be read from the Newton diagram. Let us consider the bisector of the positive orthant in $\eR^n$, that is the line consisting of points with equal coordinates. The bisector intersects the boundary of the Newton polyhedron in the exactly one point $(c,\dots,c)$, which is called the center of the boundary of the Newton polyhedron. The number $c$ is called the {\em distance to the Newton polyhedron}. {\em Remoteness of the Newton polyhedron} is equal to $r=-1/c$. If $r>-1$ the Newton polyhedron is {\em remote}, which means that it does not contain the point $(1,\dots,1)$.

Let the phase be an analytic function in a neighborhood of its critical point. {\em Remoteness of the critical point} of the phase is the upper bound of remotenesses of the Newton polyhedra of the Taylor series of the phase in all systems of local analytic coordinates with the origin at the critical point. The coordinates in which the remoteness is the greatest, are called the {\em adapted} coordinates to the critical point.

We consider the open face which contains the center of the boundary of the Newton polyhedron. The codimension of this face, less one, is called the {\em multiplicity} of the remoteness. If the face is a vertex then the multiplicity is $n-1$, and if the face is an edge then the multiplicity is $n-2$.

\section{Main results}

We use Theorems 6.1., 6.2., 6.3., 6.4. from \cite{arnold-vol2} in order to measure the oscillatority of the oscillatory integral by using the box dimension. These theorems give the asymptotic expansion of (\ref{integral}) if the phase $f$ has no critical points, nondegenerate and degenerate critical points. Theorem 6.4. involves Newton diagrams. In our theorems we use these results about asymptotic expansions.

In Theorems \ref{tm_case_n1}, \ref{tm_case_n2} and \ref{tm_case_ng2} we present our main results about fractal analysis of singularities in dimensions $n=1$, $n=2$ and $n>2$, respectively. Proofs of these theorems are presented in Section \ref{section_proofs}.

\begin{theorem}[The phase function of a single variable]\label{tm_case_n1}
Let $n=1$, the standard assumptions on $f$ and $\phi$ hold, and let $f(0)\neq 0$. Assume $f'(0)=f''(0)=\cdots=f^{(s-1)}(0)=0$ and $f^{(s)}(0)\neq 0$ for some integer $s\geq 2$. Let $\Gamma$ be the curve defined by (\ref{integral}) and (\ref{curve}), near the origin. Then:

$(i)$ The oscillatory dimension of both $X$ and $Y$ from (\ref{curve}) is equal to $d'=\frac{3s-1}{2s}$ and associated graphs are Minkowski nondegenerate.

$(ii)$ The curve dimension of $I$ is $d=\frac{2s}{s+1}$, curve $\Gamma$ is Minkowski measurable, and $d$-dimensional Minkowski content of $\Gamma$ is
    \begin{equation}\label{Mcontent}
    \M^d(\Gamma)=|C_1|^{\frac{2s}{s+1}}\cdot\pi\cdot\left(\frac{\pi}{s\cdot f(0)}\right)^{-\frac{2}{s+1}}\cdot\frac{s+1}{s-1},
    \end{equation}
where the constant $C_1$ depends on the phase function $f$ and on the amplitude function value in the origin $\phi(0)$.
\end{theorem}

\remark
The constant $C_1$ can be explicitly calculated using a standard formula for phase functions with nondegenerate critical point; see Remark \ref{remark_non_deg_coeff} in Section \ref{section_examples} with examples. For a more general case of phase functions $f$ see \cite{stein}.

\begin{theorem}[The phase function of two variables]\label{tm_case_n2}
Let $n=2$, the standard assumptions on $f$ and $\phi$ hold, and let $f(0)\neq 0$.
Let $\beta$ be the remoteness of the critical point of the phase function $f$.
Let $\Gamma$ be the curve defined by (\ref{integral}) and (\ref{curve}), near the origin, with asymptotic expansion (\ref{expansion}). Then:

$(i)$ If the multiplicity of the remoteness $\beta$ is equal to $0$ or the remoteness $\beta$ is equal to $-1$, then the oscillatory dimension of both $X$ and $Y$ from (\ref{curve}) is equal to $d'=(\beta+3)/2$ and the associated graphs are Minkowski nondegenerate. The curve dimension of $I$ is $d=2/(1-\beta)$ and the associated Minkowski content is
\begin{equation}\label{Mcontent2}
    \M^d(\Gamma)=\left[\frac{|a_{0,\beta}(\phi)|}{f(0)^{\beta}}\right]^{\frac{2}{1-\beta}}\cdot[-\beta]^{\frac{2\beta}{1-\beta}} \cdot\pi^{\frac{1+\beta}{1-\beta}}\cdot\frac{1-\beta}{1+\beta}.
\end{equation}

$(ii)$ If the multiplicity of the remoteness $\beta$ is equal to $1$ and the remoteness $\beta$ is bigger than $-1$, then the oscillatory and curve dimensions are the same as in the previous case with associated degenerate Minkowski contents

\end{theorem}

\begin{theorem}[The phase function of more than two variables]\label{tm_case_ng2}
Let $n>2$ the standard assumptions on $f$ and $\phi$ hold, and let $f(0)\neq 0$.
Let the principal part of the Taylor series of $f$ at its critical point is $\eR$-nondegenerate, and the Newton polyhedron of this series is remote with the remoteness of the Newton polyhedron equal to $\beta$.
Let $\Gamma$ be the curve defined by (\ref{integral}) and (\ref{curve}), near the origin, having asymptotic expansion (\ref{expansion}). Then:

$(i)$ If $a_{0,\beta}(\phi)\ne 0$ and $a_{i,\beta}=0$, for $i=1,\dots,n-1$, the oscillatory dimension of both $X$ and $Y$ from (\ref{curve}) is equal to $d'=(\beta+3)/2$ and the associated graphs are Minkowski nondegenerate. The curve dimension of $I$ is $d=2/(1-\beta)$ and the associated Minkowski content is given by (\ref{Mcontent2}).

$(ii)$ If for some $L>0$ holds $a_{L,\beta}\neq 0$, the oscillatory and curve dimensions are the same as for the previous case and the associated Minkowski contents are degenerate.

\end{theorem}

\begin{remark}
In Theorems \ref{tm_case_n1}, \ref{tm_case_n2} and \ref{tm_case_ng2}, if we take $f(0)=0$, then the curve $\Gamma$ and the associated reflected graphs are rectifiable, and all dimensions are equal to $1$.
\end{remark}

If there are no singularities in the observed domain, which is given by the support of the amplitude $\phi$, then Proposition \ref{pr_no_sing} gives only a trivial fractal dimension.

\begin{prop} {\rm (The regular phase function)}\label{pr_no_sing}
Assume that the standard assumptions on $\phi$ hold, and that $f$ does not have any critical point contained in the interior of the support of $\phi$. Let $\Gamma$ be the curve defined by (\ref{integral}) and (\ref{curve}), near the origin. Then $\Gamma$ is a rectifiable curve and the curve dimension of $I$ is equal to $1$. Furthermore, the graphs of the functions $x(t)=X(1/t)$ and $y(t)=Y(1/t)$, where $X$ and $Y$ are from (\ref{curve}), are rectifiable.
Hence, the oscillatory dimension of $X$ and $Y$ equals $1$.
\end{prop}
\begin{proof}
From \cite[Theorem 6.1]{arnold-vol2} it follows that $I(\tau)$ tends to zero more rapidly than any power of the parameter, as $\tau\to+\infty$.
The claim is based on the Riemann-Lebesgue lemma, see \cite[p.\ 16]{wong}.
For a $1$-dimensional situation we have
\[
I'(\tau)=i\int_{\eR}e^{i\tau f(x)}f(x)\phi(x) dx.
\]
The integral  $I'(\tau)$ admits the same type of asymptotic expansion as $I(\tau)$ and
all derivatives go to zero more rapidly than any power, that is,  $\tau^n I^{(k)}(\tau)\to 0$ as $\tau\to+\infty$, for all $k\ge 0$, $n\in\eN$.

We deduce that $\tau^{n}\sqrt{X'\left(\tau\right)^{2}+Y'\left(\tau\right)^{2}}\rightarrow0$
as $\tau\rightarrow+\infty$, for all $n\in\eN$,
so that $\Gamma$ is rectifiable. Therefore
(see \cite{tricot}) its box dimension equals $1$. For the same reason
$x'\left(t\right)=-t^{-2}X'\left(1/t\right)\rightarrow0$ as $t\rightarrow0$,
hence $\int_{0}^{x_{0}}\sqrt{1+x'\left(t\right)^{2}}dt<\infty$. The
same holds for $y$. It proves that graphs of functions $x$ and $y$
are rectifiable, so the oscillatory dimension of $X$ and $Y$ equals
$1$.
\end{proof}

Proposition \ref{pr_nondeg_high} demonstrates that in dimensions higher than $2$, nondegenerate singularities cannot be detected by the fractal dimension.

\begin{prop} {\rm (The nondegenerate critical point in a higher dimension)}\label{pr_nondeg_high}
Assume that the standard assumptions on $\phi$ and $f$ hold, and that $0\in\eR^n$, where $n>2$, is a nondegenerate critical point of $f$ (the Hessian matrix of $f$ is not equal to zero). Let $\Gamma$ be the curve defined by (\ref{integral}) and (\ref{curve}), near the origin. Then $\Gamma$ is a rectifiable curve and the curve dimension of $I$ is equal to $1$. Furthermore, the graphs of the functions $x(t)=X(1/t)$ and $y(t)=Y(1/t)$, where $X$ and $Y$ are from (\ref{curve}), are rectifiable.
Hence, the oscillatory dimension of $X$ and $Y$ equals $1$.
\end{prop}
\begin{proof}
The nondegeneracy of the critical point implies that $I\left(\tau\right)\sim C\cdot e^{i\tau f\left(0\right)}\cdot\tau^{-\frac{n}{2}}$
as $\tau\rightarrow+\infty$, where $C\in\Ce$ (see \cite[Theorem 6.2]{arnold-vol2}.
As in the proof of Proposition \ref{pr_no_sing}, $I'\left(\tau\right)$ admits
the same type of asymptotic expansion. Hence $\sqrt{X'\left(\tau\right)^{2}+Y'\left(\tau\right)^{2}}\leq C_1\tau^{-\frac{n}{2}}$
for some $C_1>0$, so $\Gamma$ is rectifiable.

As above, $x'\left(t\right)^{2}=t^{-4}X'\left(\frac{1}{t}\right)^{2}\leq C_2 t^{n-4}$ for some $C_2>0$.
So $\sqrt{1+x'\left(t\right)^{2}}\leq 1 + C_{3}t^{\frac{n}{2}-2}$ for some $C_{3}>0$. As $\frac{n}{2}>1$, we conclude that the graph of
$x$ is rectifiable. The same holds for $y$. About the dimensions, we conclude as in the
proof of Proposition \ref{pr_no_sing}.
\end{proof}

\section{Examples}\label{section_examples}

\begin{remark}\label{remark_non_deg_coeff}
In \cite[Theorem 6.2]{arnold-vol2} there is an explicit formula for the leading coefficient in the asymptotic expansion of the oscillatory integral with a nondegenerate critical point  of the phase in space of the dimension $n$. If the phase $f$ and the amplitude $\phi$ satisfy the standard assumptions, then a leading coefficient is the coefficient of the power ${\tau}^{-n/2}$ and is equal to
$$
\phi(0)(2\pi)^{n/2}\exp \left((i\pi/4)\cdot\sgn(f_{xx}^{\prime\prime} (0))\right)|\det f_{xx}^{\prime\prime} (0)|^{-1/2}.
$$
\end{remark}

\begin{example} A computation of the Minkowski content of the curve for the nondegenerate case in $1$-dimensional space.

Using Theorem \ref{tm_case_n1}, for $s=2$ we obtain oscillatory and curve dimensions for the integral and the curve defined by (\ref{integral}) and (\ref{curve}), respectively. The oscillatory dimension is equal to $5/2$ and the curve dimension is equal to $4/3$.  Using Remark \ref{remark_non_deg_coeff}, for $n=1$ we compute
$$ 
C_1=\phi(0)\sqrt{2\pi}\,{| f^{\prime\prime} (0)|}^{-1/2}\exp \left((i\pi/4)\cdot\sgn(f^{\prime\prime} (0)) \right) ,
$$
and using formula (\ref{Mcontent}) we obtain the Minkowski content of the curve $\Gamma$
$$
\M^{4/3}(\Gamma)=3|C_1|^{\frac{4}{3}}\pi\left(\frac{\pi}{2 f(0)}\right)^{-\frac{2}{3}}.
$$
For an example, if $f(x)=x^2+1$, then we have
$$
\M^{4/3}(\Gamma)=3\cdot 2^{2/3}\pi {\phi(0)}^{4/3}.
$$

\end{example}

\begin{example} A caustic consisting of the elementary critical points $A_k$ and $D_k$. 

\cite{arnold-vol1} and \cite{arnold-vol2} introduced the classification of singularities using normal forms of singularities and parametric families.  According to the assumptions of our theorems here we work with maps whose critical point does not coincide with the zero point, so we shift the graph of our map. The situation when these points coincide is not oscillatory, see expansion (\ref{expansion}) for $f(0)=0$, so we take $f(0)=1$.   Let us suppose that for a given value of the parameters, the phase function has a unique critical point. 
In this case the caustic in a neighborhood of the given value of the parameter is said to be elementary. Here we mention examples of elementary caustics obtained   by varying two or three parameters, \cite[p.\ 174, 185]{arnold-vol2}, \cite[p.\ 246]{arnold-vol1}. The caustics consist of degenerate critical points of type $A_k$ for $k\ge 1$, and $D_k$ for $k\ge 4$. Contributions of the critical points of the phase to the asymptotic expansion of oscillatory integrals depend on the type of these critical points. Each degenerate critical point has contribution of order $\tau^{\gamma-n/2}$, where $\gamma=(k-1)/(2k+2)$ for $A_k$, and $\gamma=(k-2)/(2k-2)$ for $D_k$, which are singular indices. According to Theorem \ref{tm_case_n2}, the box dimension of the associated curve, the curve dimension, is equal to $d=2/(1-\beta)$, where $\beta=\gamma-n/2$. If $k\to\infty$ then $\gamma\to 1/2$, so $\beta\to(1-n)/2$, hence the curve dimension $d\to 4/(1+n)$. We see that the curve dimension increases and tends to $2$ for $n=1$, and tends to $4/3$ for $n=2$, when we have more complicated critical points whose singular index tends to $1/2$. The oscillatory dimension is equal to $d'=\frac{3+\beta}{2}$.
\end{example}

\begin{example}\label{example_greenblatt} The normal forms of the type $x^p+y^q$.

Consider the phase $f(x,y)=x^p+y^q+1$, for integers $p,q\geq 2$ and $(p,q)\neq (2,2)$, so that $f(0,0)\neq 0$. In this case the remoteness $\beta=-\frac1p-\frac1q$, hence it follows from Theorem \ref{tm_case_n2} that the oscillatory dimension is equal to
$$
d^{\prime} =\frac{2}{1+\frac1p +\frac1q} ,
$$
while the curve dimension is equal to
$$
d=\frac32-\frac{1}{2p}-\frac{1}{2q} .
$$
The computation of Minkowski content (\ref{Mcontent2}) is more involved, as it depend on the computation of the first coefficient in asymptotic expansion (\ref{expansion}) of the integral. Notice that in this example we replace our standard notation $\Gamma$ for the curve associated to the oscillatory integral with $\mathcal{C}$, in order to avoid a confusion with the gamma function. According to \cite{greenblatt} we can compute the first coefficient in the expansion. The phase is written in adapted coordinates, which means that the remoteness is the biggest possible, in the set of all remotenesses of Newton diagrams of the map in different coordinate systems. 
In this case the Newton diagram has only one compact side $S_0$. As the bisector intersects the interior of the compact edge, the leading term of the asymptotic expansion is $d_0(\phi)\tau ^{-\beta}$, where $\beta$ is the remoteness and $\phi$ the amplitude.
First, define the function $S_0(x,y)$ to be equal to $f(x,y)$. Now, define the function $S_0^+(x,y)^{-\frac{1}{d}}$ to be equal to $S_0(x,y)^{-\frac{1}{d}}$ when $S_0(x,y)>0$ and zero otherwise. Analogously, define the function $S_0^-(x,y)^{-\frac{1}{d}}$ to be equal to $(-S_0(x,y))^{-\frac{1}{d}}$ when $S_0(x,y)<0$ and zero otherwise. 
The coordinates are \emph{superadapted}; see \cite{greenblatt}, which means that $S_0(1,y)$ and $S_0(-1,y)$ have no real roots of order bigger than $-1/\beta$, except $y=0$. Hence, according to \cite[Theorem 1.2]{greenblatt}, if we put
$$
c_0(\phi):=\frac{\phi (0,0)}{m+1}\int_{-\infty}^{+\infty}\left(S_0^+(1,y)^{\beta}+S_0^+(-1,y)^{\beta}\right)dy ,
$$
$$
C_0(\phi):=\frac{\phi (0,0)}{m+1}\int_{-\infty}^{+\infty}\left(S_0^-(1,y)^{\beta}+S_0^-(-1,y)^{\beta}\right)dy ,
$$
where $-1/m$ is a slope of the edge $S_0$, then the leading term coefficient of the asymptotic expansion of $I(\tau)$ is equal to
\begin{equation}\label{greenblatt_coefficient}
a_{0,\beta}(\phi) = -\beta\,\Gamma\left(-\beta\right)\left(e^{-i\frac{\pi}{2}\beta}c_0(\phi)+e^{i\frac{\pi}{2}\beta}C_0(\phi)\right) .
\end{equation}

Obviously, there are four distinct cases in the computation regarding the integers $p$ and $q$ being odd or even. We will compute the leading term coefficient and the Minkowski content (\ref{Mcontent2}) for the case of $p$ and $q$ being even. The other three cases are computed in similar fashion.
We compute $\beta=-1/p-1/q$ and $m=p/q$. After integration we obtain the result
\[
  c_0(\phi)=\frac{4\,\phi(0,0)}{q(m+1)}\,B\left(\frac{1}{p}, \frac{1}{q}\right),\qquad C_0(\phi)=0 ,
\]
\[
a_{0,\beta}(\phi) = 4\,\phi(0,0)\,e^{i\frac{\pi}{2}\left(\frac{1}{p}+\frac{1}{q}\right)}\Gamma\left(\frac{1}{p}+1\right)\Gamma\left(\frac{1}{q}+1\right) ,
\]
expressed using the beta function $B$ and the gamma function $\Gamma$.
As we took $f(0,0)=1$ and as we can without loss of generality fix $\phi(0,0)=1$, we obtain the Minkowski content of the curve $\mathcal{C}$ to be equal to
\[
\M^{\frac{3+\beta}{2}}(\mathcal{C}) = \left[4\,\Gamma\left(\frac{1}{p}+1\right)\Gamma\left(\frac{1}{q}+1\right)\right]^{\frac{2}{1-\beta}}\cdot[-\beta]^{\frac{2\beta}{1-\beta}} \cdot\pi^{\frac{1+\beta}{1-\beta}}\cdot\frac{1-\beta}{1+\beta} ,
\]
by putting the coefficient $a_{0,\beta}(\phi)$ in formula (\ref{Mcontent2}), where $\beta$ depends only on $p$ and $q$. Notice that the Minkowski content depends essentially only on $p$ and $q$.

Finally, notice that the normal forms from this example include singularities of the standard classification (see \cite{arnold-vol1}) types $E_6$ and $E_8$, for $(p,q)=(3,4)$ and $(p,q)=(3,5)$, respectively. Also, the normal form for the ordinary cusp is obtained by taking $(p,q)=(2,3)$.
\end{example}

\section{Proofs of main results}
\label{section_proofs}

\begin{proof}[Proof of Theorem \ref{tm_case_n1}]
Without the loss of generality we assume that $f(0)>0$. In the case of $f(0)<0$, we consider the integral $J$ having the phase $\widetilde{f}(x)=-f(x)$. Now $J(\tau)=\overline{I(\tau)}$ and we see from the definitions of fractal properties (oscillatory and curve dimensions and Minkowski contents) of an oscillatory integral, that they are invariant to complex conjugation of $I$.

We use the asymptotic expansion of the integral $I$ from (\ref{integral}),
$$
I(\tau) \sim e^{i\tau f(0)}\sum\limits_{j=1}^{\infty} C_j\cdot\tau^{-j/s},\quad\mathrm{as}\ \tau\to\infty,
$$
where $C_j\in\Ce$, from \cite[Proposition 3 on page 334]{stein}, and it holds that $C_1\neq 0$. From the same reference, it follows that each constant $C_j$ depends on only finitely many derivatives of $f$ and $\phi$ at $0$.
We write
\begin{equation}\label{integral-eP}
I(\tau) = e^{i\tau f(0)}P(\tau),
\end{equation}
where the function $P(\tau)\sim\sum\limits_{j=1}^{\infty} C_j\cdot\tau^{-j/s}$, as $\tau\to\infty$.

First we show that the function $I$ is of class $C^{\infty}(\eR)$, using derivation under the integral sign. By taking the derivative of (\ref{integral}), we get
$$
I'(\tau)=i\int_{\eR^n}e^{i\tau f(x)}\phi_1(x) dx,
$$
where $\phi_1(x)=f(x)\phi(x)$. Inductively, we see that
$$
I^{(m)}(\tau)=i^m\int_{\eR^n}e^{i\tau f(x)}\phi_m(x) dx,\quad\mathrm{for}\ \mathrm{all}\ m\in\eN,
$$
where $\phi_m(x)=[f(x)]^m\phi(x)$. Notice that for every $m\in\eN$, the function $I^{(m)}$ is equal to the constant $i^m$ multiplying the oscillatory integral of type (\ref{integral}), with the phase $f$ and the amplitude $\phi_m$. Further, using the asymptotic expansion of this integral, we get
\begin{equation}\label{D-integral-eP}
I^{(m)}(\tau)=i^m e^{i\tau f(0)}P_m(\tau),\quad\mathrm{for}\ \mathrm{all}\ m\in\eN,
\end{equation}
where the function $P_m$ possesses an asymptotic expansion in the same asymptotic sequence as $P$.

Now we want to prove that the function $P$ is of class $C^{\infty}(\eR)$, and that its derivative of any order possesses an asymptotic expansion in the same asymptotic sequence as $P$, that is, $P^{(m)}(\tau)\sim\sum\limits_{j=1}^{\infty} C_j^{(m)}\cdot\tau^{-j/s}$, as $\tau\to\infty$, where $C_j^{(m)}\in\Ce$. Notice that from (\ref{integral-eP}) and the fact that $I\in C^{\infty}$ immediately follows that $P\in C^{\infty}$.

By taking the derivative of (\ref{integral-eP}), we get
$$
I'(\tau)=i f(0) e^{i\tau f(0)}P(\tau)+e^{i\tau f(0)}P'(\tau).
$$
Respecting (\ref{D-integral-eP}) and dividing every term by $e^{i\tau f(0)}$, we get the expression
$$
P'(\tau)=i\left(P_1(\tau)-f(0)P(\tau)\right),
$$
using \cite[p.\ 14]{erdelyi}, it shows that $P'$ possesses an asymptotic expansion in the same asymptotic sequence as $P$. It follows by induction, that for all $m\in\eN$, the function $P^{(m)}$ also possesses an asymptotic expansion in the same asymptotic sequence as $P$.

Notice that the exponents of the monomials of the asymptotic sequence are integer multiples of  a common real number $-1/s$. Hence, it follows from the clasical proof; see \cite[p.\ 21]{erdelyi}, that the asymptotic expansion of $P^{(m)}$ is given by $m$ times differentiating the asymptotic expansion of $P$, term by term.

Now define $a_j=Re\ C_j$ and $b_j=Im\ C_j$ for all $j\in\eN$. Also, define functions $A(\tau)=Re\ P(\tau)$ and $B(\tau)=Im\ P(\tau)$. Respecting (\ref{curve}) we get
\begin{eqnarray}
X(\tau) &=& \cos(\tau f(0))A(\tau)-\sin(\tau f(0))B(\tau),\\
Y(\tau) &=& \sin(\tau f(0))A(\tau)+\cos(\tau f(0))B(\tau),
\end{eqnarray}
where
\begin{eqnarray}
A(\tau) &\sim& \sum\limits_{j=1}^{\infty} a_j\cdot\tau^{-j/s},\quad\mathrm{as}\ \tau\to\infty,\\
B(\tau) &\sim& \sum\limits_{j=1}^{\infty} b_j\cdot\tau^{-j/s},\quad\mathrm{as}\ \tau\to\infty.
\end{eqnarray}
Notice that from $P(\tau)=A(\tau)+i B(\tau)$ follows that functions $A$ and $B$ are of class $C^{\infty}(\eR)$ and that $A^{(m)}$ and $B^{(m)}$, $m\in\eN_0$, possess asymptotic expansions given by $m$ times differentiating the asymptotic expansions of $A$ and $B$, term by term, respectively.

For the oscillatory dimension we will provide the proof for $Y$. For the function $X$ the proof is analogous. Using the substitution $t=1/\tau$, to determine the oscillatory dimension of $Y$, we further investigate the asymptotic expansion of the function $y$, defined by $y(t)=Y(1/t)$, near the origin,
$$
y(t)=\sin(f(0)/t)a(t)+ \cos(f(0)/t)b(t),
$$
where $a(t)=A\left(t^{-1}\right)$ and $b(t)=B\left(t^{-1}\right)$. Now
\begin{eqnarray}
a(t) &\sim& \sum\limits_{j=1}^{\infty} a_j\cdot t^{j/s},\quad\mathrm{as}\ t\to 0^+,\\
b(t) &\sim& \sum\limits_{j=1}^{\infty} b_j\cdot t^{j/s},\quad\mathrm{as}\ t\to 0^+.
\end{eqnarray}
Notice that both functions $a$ and $b$ are of class $C^{\infty}(\eR^+)$ and that both $a^{(m)}$ and $b^{(m)}$, for all $m\in\eN_0$, possess asymptotic expansions, near the origin, given by $\displaystyle\frac{d^m}{d t^m}\left[A\left(t^{-1}\right)\right]$ and $\displaystyle\frac{d^m}{d t^m}\left[B\left(t^{-1}\right)\right]$, respectively. Indeed, this $m$-th derivatives are finite linear combinations of products given by $A^{(k)}\left(t^{-1}\right)$ or $B^{(k)}\left(t^{-1}\right)$, multiplied by a negative power of $t$, where $k\leq m$. A linear combination of asymptotic expansions is again an asymptotic expansion; see \cite[p.\ 14]{erdelyi}.

Finally, we define functions $p(t)=\sqrt{a^2(t)+b^2(t)}$ and $\psi:\eR\rightarrow [0,2\pi)$ such that
$$
\cos\psi(t)=\frac{a(t)}{p(t)},\qquad \sin\psi(t)=\frac{b(t)}{p(t)}.
$$
Exploiting trigonometric addition formulas we get the expression
$$
y(t)=p(t)\sin(q(t)),\quad\mathrm{where}\ q(t)=f(0)\cdot t^{-1}+\psi(t).
$$
The function $p$ is of class $C^{\infty}(\eR^+)$, and $q$ is also $C^{\infty}(\eR^+)$, as $\psi$ is $C^{\infty}(\eR^+)$ by the definition, for sufficiently small $t$. For the derivative of $q$, we get
\begin{equation}\label{Der-q}
q'(t)=-f(0)\cdot t^{-2}+\psi'(t),\quad\mathrm{where}\ \psi(t)=\arctan\frac{b(t)}{a(t)} .
\end{equation}
Because $b(t)/a(t)\to\mathrm{const}$ as $t\to\infty$, and as $\arctan$ is an analytic function, then for all $m\in\eN$, $\psi^{(m)}(t)$ and $q^{(m)}$ possess an asymptotic expansion. Using the same principle, we can prove that for all $m\in\eN$, $p^{(m)}$ possesses an asymptotic expansion.

The asymptotic representation of the function $p$ is easily determined,
\begin{eqnarray}
p(t) &=& \sqrt{\left(a_1 t^{1/s} + O\left(t^{2/s}\right)\right)^2 + \left(b_1 t^{1/s} + O\left(t^{2/s}\right)\right)^2}\\
     &=& t^{1/s}\sqrt{a^2_1+b^2_1}\left(1+ O\left(t^{1/s}\right)\right) \sim t^{1/s}\sqrt{a^2_1+b^2_1},
\end{eqnarray}
as $t\to 0$, and the asymptotic representation of derivative of any order of $p$ is given by differentiating that many times the asymptotic representation of $p$.
For the function $q$, as $\psi$ is bounded, it follows that $q(t)\sim f(0)\cdot t^{-1}$, as $t\to 0$, and for all $m\in\eN$, $q^{(m)}\sim f(0)\cdot\frac{d^m}{dt^m}\left[t^{-1}\right]$, as $t\to 0$.

Finally, we use Theorem \ref{BDchirp} from Section \ref{section_pwr-log-chrips-spirals}, with $S(t)=\sin t$, and constants $T=\pi$, $\alpha=1/s$ and $\beta=1$. Notice that all of the assumptions of that theorem are satisfied. Let $\Gamma_y$ be the graph of the function $y$. We conclude that $\dim_B\Gamma_y=\dim_{osc}Y=2-\frac{\alpha+1}{\beta+1}=\frac{3s-1}{2s}$ and that $\Gamma_y$ is Minkowski nondegenerate.

\smallskip

In order to compute the curve dimension, we want to investigate the oscillatory integral $I$ in polar coordinates. We first define the real function $G(\tau)=|I(\tau)|=|P(\tau)|=\sqrt{X^2(\tau)+Y^2(\tau)}=\sqrt{A^2(\tau)+B^2(\tau)}$, hence it follows that $G$ is of class $C^{\infty}$. Using asymptotic expansions for $A$ and $B$, we get
$$
G(\tau)\sim\sum\limits_{j=1}^{\infty} c_j\cdot\tau^{-j/s},\qquad\mathrm{as}\ \tau\to\infty,
$$
where $c_j=\sqrt{a^2_j+b^2_j}=|C_j|\in\eR$, for all $j\in\eN$. Next, we define the continuous function $\varphi:[\tau_0,\infty)\rightarrow\eR$, $\tau_0>0$, by
$$
\tan\varphi(\tau)=\frac{Y(\tau)}{X(\tau)},
$$
where $X(\tau)\neq 0$, and extend it by continuity. The zero set of $X$ is discrete because of the asymptotic expansion of $X'(\tau)$. Using trigonometric addition formulas we calculate
$$
\tan\varphi(\tau)=\frac{\sin(\tau f(0)+\Psi(\tau))G(\tau)}{\cos(\tau f(0)+\Psi(\tau))G(\tau)}=\tan(\tau f(0)+\Psi(\tau)),
$$
where $\Psi:[\tau_0,\infty)\rightarrow[0,2\pi)$ is such that
$$
\cos\Psi(\tau)=\frac{A(\tau)}{G(\tau)},\qquad \sin\Psi(\tau)=\frac{B(\tau)}{G(\tau)}.
$$
We compute the expression
$$
\Psi'(\tau)=\frac{A(\tau)B'(\tau)-B(\tau)A'(\tau)}{A^2(\tau)+B^2(\tau)}=K\cdot\tau^{-1-1/s}\left(1+O\left(\tau^{-1/s}\right)\right),\ \mathrm{as}\ \tau\to\infty,
$$
where the constant $\displaystyle K=\frac{a_2 b_1-a_1 b_2}{s\left(a_1^2+b_1^2\right)}$, so $\varphi'(\tau)=f(0)+\Psi'(\tau)\sim f(0)+K\cdot\tau^{-1-1/s}$, as $\tau\to\infty$. From the expression for $\Psi'(\tau)$ it follows that $\Psi'$ is of class $C^{\infty}$. As $\Psi$ is a continuous function for sufficiently large $\tau_0$, it follows that $\Psi$ is of class $C^{\infty}$ and it holds $\varphi(\tau)=\tau f(0)+\Psi(\tau)$, hence $\varphi$ is also of class $C^{\infty}$. Analogously as before, functions $G^{(m)}$, $\varphi^{(m)}$ and $\Psi^{(m)}$ possess asymptotic expansions for all $m\in\eN_0$.  As $f(0)>0$, we can take $\tau_0>0$ sufficiently large such that $\varphi'(\tau)>0$, for every $\tau\in[\tau_0,\infty)$. As now $\varphi:[\tau_0,\infty)\rightarrow[\varphi_0,\infty)$, where $\varphi_0=\varphi(\tau_0)$, is of class $C^{\infty}$ and a strictly increasing bijection, so is its inverse function $\tau:[\varphi_0,\infty)\rightarrow[\tau_0,\infty)$.

Now the radius function $r:[\varphi_0,\infty)\rightarrow[0,\infty)$, defined by $r(\varphi)=G(\tau(\varphi))$, is of class $C^{\infty}$. We want to determine asymptotic representations of two derivatives of $r(\varphi)$, as $\varphi\to\infty$. From before, we know that $G(\tau)\sim c_1\cdot\tau^{-1/s}$, $G'(\tau)\sim -\frac{c_1}{s}\cdot\tau^{-1-1/s}$ and $G''(\tau)\sim \frac{c_1}{s}\left(1+\frac{1}{s}\right)\cdot\tau^{-2-1/s}$, as $\tau\to\infty$.
We compute
\begin{eqnarray}
r'(\varphi)  &=& G'(\tau(\varphi))\tau'(\varphi)=       \frac{G'(\tau(\varphi))}{\varphi'(\tau(\varphi))}, \nonumber\\
r''(\varphi) &=& G''(\tau(\varphi))(\tau'(\varphi))^2 + G'(\tau(\varphi))\tau''(\varphi) \nonumber\\
&=& \frac{G''(\tau(\varphi))}{(\varphi'(\tau(\varphi)))^2}-G'(\tau(\varphi))\frac{\varphi''(\tau(\varphi))}{\left[\varphi'(\tau(\varphi))\right]^3} , \nonumber
\end{eqnarray}
since $\tau'(\varphi)=\left[\varphi'(\tau(\varphi))\right]^{-1}$ and $\tau''(\varphi)=-\varphi''(\tau(\varphi))/\left[\varphi'(\tau(\varphi))\right]^3$.

As $\Psi$ is a bounded function, it follows that $\varphi(\tau)\sim\tau f(0)$, as $\tau\to\infty$. It is easy to see that the inverse $\tau(\varphi)\sim\varphi/f(0)$, as $\varphi\to\infty$.
Notice that $\varphi'(\tau)\sim f(0)$ and $\varphi''(\tau)\sim K\left(-1-\frac{1}{s}\right)\cdot\tau^{-2-1/s}$, as $\tau\to\infty$. Finally, notice that $\tau(\varphi)\to\infty$, as $\varphi\to\infty$.
Hence, we can compute
\begin{eqnarray}
r(\varphi)  &\sim& c_1\cdot\left(\varphi/f(0)\right)^{-1/s} = c_1 f(0)^{1/s}\varphi^{-1/s}, \nonumber\\
r'(\varphi) &\sim& \frac{-\frac{c_1}{s}\cdot\left(\varphi/f(0)\right)^{-1-1/s}}{f(0)} = -\frac{c_1}{s}f(0)^{1/s}\cdot\varphi^{-1-1/s}, \nonumber\\
r''(\varphi)&\sim& \frac{\frac{c_1}{s}\left(1+\frac{1}{s}\right)\cdot\left(\varphi/f(0)\right)^{-2-1/s}}{(f(0))^2} \nonumber\\
&+& \frac{c_1}{s}\cdot\left(\varphi/f(0)\right)^{-1-1/s}\frac{K\left(-1-\frac{1}{s}\right)\cdot\left(\varphi/f(0)\right)^{-2-1/s}}{(f(0))^3} \nonumber\\
            &\sim& \frac{c_1}{s}\left(1+\frac{1}{s}\right)f(0)^{1/s}\cdot\varphi^{-2-1/s}, \nonumber
\end{eqnarray}
as $\varphi\to\infty$. Notice, as $c_1>0$ that $r'(\varphi)<0$, for $\varphi$ sufficiently large, so we can take $\tau_0>0$ sufficiently large such that $r$ is a strictly decreasing function.

Finally, we use Theorem \ref{tm_ffam} from Section \ref{section_pwr-log-chrips-spirals}, taking $\alpha=1/s$. Function $r$ satisfies the assumptions of this theorem. We calculate the constant $m$ from (\ref{lim}), below, to be equal to $f(0)^{1/s} |C_1|$, and $|r''(\f)\,\f^{\a}|\to 0$, as $\varphi\to\infty$, so it is uniformly bounded as a function of $\varphi$ on its domain $[\varphi_0,\infty)$. As all of the assumptions of that theorem are satisfied, we conclude that the curve dimension of $I$ is $d:=2/(1+\a)=2s/(s+1)$, the curve $\Gamma$ is Minkowski measurable and its $d$-dimensional Minkowski content is given by (\ref{Mcontent}).
\end{proof}

\begin{proof}[Proof of Theorem \ref{tm_case_n2}]
Using \cite[Theorem 6.5]{arnold-vol2} we conclude that the oscillation index of the critical point of $f$ equals to its remoteness $\beta$. Then using $a_{k,\gamma}:=a_{k,\gamma}(\phi)$ and rewriting (\ref{expansion}) we get the asymptotic expansion
\begin{equation}\label{neq2}
I(\tau)\sim e^{i\tau f(0)}\left(a_{1,\beta}\tau^{\beta}\log\tau+a_{0,\beta}\tau^{\beta} + \sum_{\alpha<\beta}\left(a_{1,\alpha}\tau^{\alpha}\log\tau+a_{0,\alpha}\tau^{\alpha}\right)\right)
\end{equation}
as $\tau\to\infty$, where $\alpha$ runs through a finite set of arithmetic progressions, hence there exists $\varepsilon$ such that $|\beta-\alpha|>\varepsilon$ for all such $\alpha$.

Without loss of generality, we can assume that we work in \emph{superadapted}, hence adapted coordinates; see \cite[Section 7.]{greenblatt}.
We now have to establish, for cases $(i)$ and $(ii)$, if the first coefficient $a_{1,\beta}$ is vanishing or not.

For case $(i)$, as the multiplicity of the remoteness is equal to $0$ and the dimension $n=2$, we conclude that the open face of the Newton diagram of the phase $f$ that contains the center of the boundary of the associated Newton polyhedron is an edge. If the Newton polyhedron is remote, that is $\beta>-1$, we are in the Case $1$ or $3$ from \cite[Theorem 1.2]{greenblatt}, from which it follows that $a_{1,\beta}=0$. From the definition of the oscillation index it follows that $a_{0,\beta}\neq 0$.
If $\beta=-1$, it follows from \cite[lemma 1.0]{greenblatt} that the critical point of $f$ at the origin in nondegenerate. Now, from \cite[Theorem 6.2]{arnold-vol2} it follows that $a_{0,\beta}\neq 0$ and $a_{1,\beta}=0$.

The rest of the proof now basically follows the proof of Theorem \ref{tm_case_n1}. Minor differences arise regarding treatment of more complicated asymptotic expansion (\ref{neq2}), which has terms having a logarithm function.

For case $(ii)$, the multiplicity of the remoteness is equal to $1$, hence the center of the boundary of the associated Newton polyhedron is a vertex. As $\beta>-1$, we are in the Case $2$ from \cite[Theorem 1.2]{greenblatt}, hence from \cite[Comment 2.]{greenblatt} it follows that $a_{1,\beta}\neq 0$. Now the first term in the asymptotic expansion has a logarithm inside. Like in the case $(i)$, the proof of Theorem \ref{tm_case_n1} is once more adapted concerning log-terms in asymptotic expansions. Further differences arise in the final steps of the proof, when applying Theorems \ref{BDchirp} and \ref{tm_ffam}, for oscillatory and curve dimensions, respectively.

We first consider the proof for the oscillatory dimension. It is easy to see that here, contrary to the proof of Theorem \ref{tm_case_n1}, it holds $p(t)\sim const\cdot t^{-\alpha}\log(t^{-1})$, as $t\to 0$. It follows $p'(t)\sim const\cdot t^{-\alpha-1}\log(t^{-1})$, as $t\to 0$. So instead of Theorem \ref{BDchirp} from Section \ref{section_pwr-log-chrips-spirals}, which can not be applied here, we use Theorem~\ref{BDchirp_mod}. For the proof for the curve dimension, instead of using Theorem \ref{tm_ffam} on the curve radius function $r=r(\f)$ (see the proof of Theorem \ref{tm_case_n1}), we use directly Theorem \ref{ffa}.

\end{proof}

\begin{proof}[Proof of Theorem \ref{tm_case_ng2}]
Using \cite[Theorem 6.4]{arnold-vol2} we conclude that the oscillation index of the critical point of $f$ equals to the remoteness $\beta$. Using (\ref{expansion}) we get the asymptotic expansion
\begin{equation}\label{ng2}
I(\tau)\sim e^{i\tau f(0)}\left(\sum_{k=0}^{n-1} a_{k,\beta}(\phi)\tau^{\beta}\log^k\tau + \sum_{\alpha<\beta}\sum_{k=0}^{n-1} a_{k,\alpha}(\phi)\tau^{\alpha}\log^k\tau\right)
\end{equation}
as $\tau\to\infty$, where $\alpha$ runs through a finite set of arithmetic progressions. The rest of the proof is analogous to the proof of Theorem \ref{tm_case_n2}, in both cases. Notice that here the asymptotic scale is involving terms consisting of $\tau$ to a negative rational power multiplied by a logarithm of $\tau$ to the power $k$.
\end{proof}

\section{Fractal properties of chirps and spirals related to oscillatory integrals}\label{section_pwr-log-chrips-spirals}

In order to compute curve and oscillatory dimensions of oscillatory integrals and related Minkowski contents, we use theorems presented in this section. Theorems \ref{BDchirp} and \ref{tm_ffam}, cited below, were used before in different setting, related to fractal analysis of differential equations, Fresnel integrals and dynamical systems; see \cite{cswavy}, \cite{clothoid} and \cite{zuzu}. Here, they are used in the proofs of Theorems \ref{tm_case_n1}, \ref{tm_case_n2} and \ref{tm_case_ng2}. Also, for the proofs of Theorems \ref{tm_case_n2} and \ref{tm_case_ng2}, Theorems \ref{BDchirp} and \ref{tm_ffam} had to be modified, as the original versions fail to take into account the power-log asymptotic of the leading term in the asymptotic expansion of related oscillatory integrals. This modified variants, Theorems \ref{BDchirp_mod} and \ref{ffa} below, are proved throughout the rest of this section.

\begin{theorem}[Theorem 5 from \cite{cswavy}]\label{BDchirp}
Let $y(x)=p(x)S(q(x))$, where $x \in I=(0,c]$ and $c>0$.
Let the functions $p(x)$, $q(x)$ and $S(t)$ satisfy the following assumptions:
\begin{equation}
\mbox{ $p\in C(\bar{I})\cap C^{1}(I)$, $q\in C^{1}(I)$, $S\in C^1(\eR)$}.
\end{equation}
The function $S(t)$ is assumed to be a $2T$-periodic real function defined on  $\eR$  such that
\begin{equation}\label{Sx1}
\left\{
\begin{array}{c}
\mbox{$S(a)=S(a+T)=0$ for some $a\in\eR$,}\\
\mbox{$S(t)\neq 0$ for all $t\in (a,a+T)\cup (a+T,a+2T)$,}
\end{array}
\right.
\end{equation}
where $T$ is a positive real number and  $S(t)$ alternately changes a sign on intervals $(a+(k-1)T,a+kT)$, for $k\in \eN$. Without loss of generality, we take  $a=0$.
Let us suppose that $0<\alpha \leq \beta$ and:
\begin{equation}\label{apsp_org}
p(x)\simeq_1 x^{\alpha} \quad \mbox{as} \quad x\to 0,\qquad q(x)\simeq_1 x^{-\beta} \quad \mbox{as} \quad x\to 0.
\end{equation}

Let $\Gamma_y$ be the graph of the function $y$.
Then $\dim_B\Gamma_y=2-(\alpha+1)/(\beta+1)$ and $\Gamma_y$ is Minkowski nondegenerate.
\end{theorem}

\begin{theorem}[Theorem 2 from \cite{clothoid}]\label{tm_ffam}
Assume that $\f_1>0$ and $r\:[\f_1,\ty)\to(0,\ty)$ is a decreasing  $C^2$
function converging to zero as  $\f\to\ty$. Let the limit
\bgeq\label{lim}
m:=\lim_{\f\to\ty}\frac{r'(\f)}{(\f^{-\a})'}
\endeq
exist, where $\a\in(0,1)$.
Assume that  $|r''(\f)\,\f^{\a}|$ is uniformly bounded as a function of $\f$. Let $\C$ be the graph of the
spiral $\rho=r(\f)$ and define $d:=2/(1+\a)$.
Then $\dim_B\C=d$, the spiral is Minkowski measurable, and
moreover, \bgeq
\M^d(\C)=m^d\pi(\pi\a)^{-2\a/(1+\a)}\frac{1+\a}{1-\a}.
\endeq
\end{theorem}

\begin{remark}
Theorem \ref{tm_ffam} is the simplified but equivalent form of the result first introduced in \cite{zuzu}.
\end{remark}

\begin{theorem}\label{BDchirp_mod} {\bf (Box dimension and Minkowski degeneracy of the graph of a logarithmic $(\alpha,1)$-chirp-like function)}
Let $y(x)=p(x)\sin(q(x))$, $x \in I=(0,c],$ $c>0.$
Let the functions $p(x)$ and $q(x)$ satisfy the following assumptions:
\begin{equation}\label{px}
\mbox{ $p\in C(\bar{I})\cap C^{1}(I)$, $q\in C^{1}(I)$}.
\end{equation}
Let us suppose that $0<\alpha \leq 1$, $l\in\eN$ and:
\begin{equation}\label{apsp}
p(x)\simeq_1 x^{\alpha}\left[\log(x^{-1})\right]^l \quad \mbox{as} \quad x\to 0,
\end{equation}
\begin{equation}\label{qq}
q(x)\simeq_1 x^{-1} \quad \mbox{as} \quad x\to 0.
\end{equation}

Let $\Gamma_y$ be the graph of the function $y$.
Then $\dim_B\Gamma_y=d$, where $d=2-(\alpha+1)/2$, and $\Gamma_y$ is Minkowski degenerate, having $\mathcal{M}^d(\Gamma_y)=\infty$.
\end{theorem}

\begin{remark}
Theorem \ref{BDchirp_mod} is a modified variant of Theorem \ref{BDchirp}, by setting $S(x)=\sin x$, $T=\pi$ and $\beta=1$ in the original theorem, and adapting condition (\ref{apsp_org}) by introduction of log-term asymptotics. The same applies also for Propositions \ref{druga_mod} and \ref{treca_mod}, below.
\end{remark}

The proof of Theorem \ref{BDchirp_mod} uses Theorem \ref{mersat_mod}, below, which is a modified variant of \cite[Theorem 2.1.]{mersat}, and two propositions concerning the properties of functions $p$ and $q$, which are also modified variants of \cite[Proposition 1 and 2]{cswavy}, below. We also need \cite[Definition 2.1.]{mersat}, stating that for some $\varepsilon_0>0$, we say that a function $k=k(\varepsilon)$ is \emph{an index function} on $(0,\varepsilon_0]$ if $k:(0,\varepsilon_0]\to\eN$, $k(\varepsilon)$ is nonincreasing and $\lim_{\varepsilon\to 0} k(\varepsilon)=\infty$.

\begin{theorem}[Modification of Theorem 2.1.\ from \cite{mersat}]{\label{mersat_mod}} Let $y\in C^1((0,T])$ be a bounded function on $(0,T]$. Let $s\in[1,2)$ be a real number, let $l\in\eN$ and let $(a_n)$ be a decreasing
sequence of consecutive zeros of $y(x)$ in $(0,T]$ such that $a_n\to 0$ when $n\to \infty$ and let there exist constants $c_1, c_2, \varepsilon_0$ such that for all $\varepsilon \in(0,\varepsilon_0)$ we have:
\begin{equation}\label{left}
c_1\varepsilon^{2-s}\left[\log(\varepsilon^{-1})\right]^l\leq \sum_{n\geq k(\varepsilon)} \max_{x\in[a_{n+1},a_n]} |y(x)|(a_n-a_{n+1}),
\end{equation}
\begin{equation}\label{right}
a_{k(\varepsilon)} \sup_{x\in(0,a_{k(\varepsilon)}]} |y(x)| + \varepsilon \int_{a_{k(\varepsilon)}}^{a_1} |y'(x)|dx \leq c_2\varepsilon^{2-s}\left[\log(\varepsilon^{-1})\right]^l,
\end{equation}
where $k(\varepsilon)$ is an index function on $(0,\varepsilon_0]$ such that
$$
|a_n-a_{n+1}|\leq \varepsilon \quad \mbox{for all} \quad n\geq k(\varepsilon) \quad\mbox{and}\quad \varepsilon \in (0,\varepsilon_0).
$$

Let $G(y)$ be the graph of the function $y$.
Then $\dim_B(G(y))=s$ and $G(y)$ is Minkowski degenerate, having $\mathcal{M}^s(G(y))=\infty$.
\end{theorem}

\begin{proof}
Let $G_{\e}(y)$ be the $\e$-neighbourhood of the graph $G(y)$ of the function $y$.  From \cite[Lemma 2.1.]{mersat} it follows that $|G_{\e}(y)|\geq c_1\varepsilon^{2-s}\left[\log(\varepsilon^{-1})\right]^l$, and from \cite[Lemma 2.2.]{mersat} it follows that $|G_{\e}(y)|\leq c\left[\e + c_2\varepsilon^{2-s}\left[\log(\varepsilon^{-1})\right]^l\right]$, where $c>0$. From the definitions of $\M^{*s}(G(y))$ and $\M_*^{s}(G(y))$ it follows that
\[
\M^{*s}(G(y)) \geq \M_*^{s}(G(y)) \geq \liminf_{\e\to0} \frac{c_1\varepsilon^{2-s}\left[\log(\varepsilon^{-1})\right]^l}{\e^{2-s}} = +\infty ,
\]
and that
\[
\M_*^{s'}(G(y)) \leq \M^{*s'}(G(y)) \leq \liminf_{\e\to0} \frac{c\left[\e + c_2\varepsilon^{2-s}\left[\log(\varepsilon^{-1})\right]^l\right]}{\e^{2-s'}} = 0
\]
holds for all $s'>s$, hence the theorem is proved.
\end{proof}

\begin{prop}[Modification of Proposition 1 from \cite{cswavy}]{\label{druga_mod}} Assume that the functions $p(x)$ and  $q(x)$ satisfy conditions {\rm (\ref{px})}, {\rm (\ref{apsp})} and {\rm (\ref{qq})}. Then
there exist $\delta_0 >0$, $l\in\eN$ and positive constants $C_1 \mbox{and} \  C_2$ such that:
               \begin{eqnarray}
C_1x^{\alpha}\left[\log(x^{-1})\right]^l\leq & p(x) & \leq C_2x^{\alpha}\left[\log(x^{-1})\right]^l, \nonumber\\
C_1x^{\alpha-1}\left[\log(x^{-1})\right]^l\leq & p'(x) & \leq C_2x^{\alpha-1}\left[\log(x^{-1})\right]^l, \nonumber\\
C_1x^{-1}\leq & q(x) & \leq C_2x^{-1}, \nonumber\\
C_1x^{-2}\leq & -q'(x) & \leq C_2x^{-2}, \nonumber
               \end{eqnarray}
                for all $x\in(0,\delta_0]$. Furthermore, there exists the inverse function $q^{-1}$ of the function $q$ defined on $[m_0,\infty)$, where $m_0=q(\delta_0)$, and it holds:
                \begin{equation*}\label{qq-1}
               q^{-1}(t) \simeq_1 t^{-1}\quad \mbox{as}\quad t\to\infty,
               \end{equation*}
               \begin{equation*}\label{qq-1st}
               C_1t^{-2}(t-s)\leq q^{-1}(s)- q^{-1}(t)\leq C_2s^{-2}(t-s), \quad m_0\leq s<t.
               \end{equation*}
\end{prop}

\begin{prop}[Modification of Proposition 2 from \cite{cswavy}]{\label{treca_mod}} For any function $q(x)$ with properties {\rm (\ref{px})} and {\rm (\ref{qq})}, we have:
\begin{itemize} \item[\rm{(i)}]
Let $a_k=q^{-1}(k\pi)$ and $s_k=q^{-1}(t_0+k\pi), \ k\in \eN$, where $t_0\in(0,\pi)$ is arbitrary. Then there exist $k_0\in \eN$ and $c_0>0$ such that $a_k\in(0,\delta_0]$, $y(a_k)=0$, $s_k\in(a_{k+1},a_k)$ for all $k\geq k_0$, $a_k\searrow 0$ as $k\to \infty$, $a_k\simeq k^{-1}$ as $k\to \infty$, and
\begin{equation*}\label{pro4}
 \max_{x\in[a_{k+1},a_k]} |y(x)| \geq c_0(k+1)^{-\alpha}\left[\log{(k+1)}\right]^l \quad \mbox{for all $k\geq k_0$, $c_0>0$},
\end{equation*}
where $y(x)=p(x)\sin(q(x))$ and the function $p(x)$ and $l\in\eN$ satisfy (\ref{apsp}).
                \item[\rm{(ii)}]
There exists $\varepsilon_0>0$ and a function $k:(0,\varepsilon_0)\to\eN$ such that
\begin{equation}\label{kaeps}
\frac{1}{\pi}\left(\frac{\varepsilon}{\pi C_2} \right)^{-\frac{1}{2}}\leq k(\varepsilon)\leq \frac{2}{\pi}\left(\frac{\varepsilon}{\pi C_2} \right)^{-\frac{1}{2}}.
\end{equation}
In particular,
$$\frac{C_1}{\pi}(k+1)^{-2}\leq a_k-a_{k+1}\leq \varepsilon ,$$
for all $k\geq k(\varepsilon)$ and $\varepsilon\in (0,\varepsilon_0)$.
\end{itemize}
\end{prop}

Proofs of Propositions \ref{druga_mod} and \ref{treca_mod} are analogous as in \cite{cswavy}.

\begin{proof}[Proof of Theorem \ref{BDchirp_mod}]
We have to check that assumptions (\ref{left}) and (\ref{right}) are satisfied.
By Proposition \ref{treca_mod} we have
\begin{align*}
\sum_{k\geq k(\varepsilon)} \max_{x\in[a_{k+1},a_k]} |y(x)|(a_k-a_{k+1}) &\geq \frac{c_0 C_1}{\pi}\sum_{k=k(\varepsilon)}^\infty (k+1)^{-\alpha-2}\left[\log (k+1)\right]^l \\
&\geq c\sum_{k=k(\varepsilon)+1}^\infty k^{-\alpha-2}\left[\log k\right]^l=ca,
\end{align*}
where the series $ a=\sum_{k=k(\varepsilon)+1}^\infty k^{-\alpha-2}\left[\log k\right]^l$ is convergent. Then, using the integral test for convergence and (\ref{kaeps}), we obtain that
\begin{align*}
ca &\geq \int_{k=k(\e)+1}^\infty k^{-\alpha-2}\left[\log k\right]^l \geq c_1(\frac{1}{k(\e)+1})^{\alpha+1}\left[\log (k(\e)+1)\right]^l \\
&\geq \frac{c_1 }{2}(\frac{1}{k(\e)})^{\alpha+1}\left[\log k(\e)\right]^l \geq c_2\e^{2-\left(2-\frac{\alpha+1}{2}\right)}\left[\log\left(\e^{-1}\right)\right]^l,\nonumber
\end{align*}
for all $\varepsilon \in (0,\varepsilon_0)$. Using Proposition \ref{druga_mod} it follows that
\begin{align*}
|y'(x)|=|p'(x)\sin(q(x))+p(x)q'(x)\cos(q(x))| &\leq c_3 x^{\alpha-1}\left[\log\left(x^{-1}\right)\right]^l \\
&\leq c_4 x^{\alpha-2},
\end{align*}
which holds near $x=0^+$. By Proposition \ref{treca_mod} we conclude that
\begin{align*}
a_{k(\e)} \sup_{x\in(0,a_{k(\e)}]} |y(x)| &+ \e \int_{a_{k(\e)}}^{a_{k_0}} |y'(x)|dx \\
&\leq c_5\e^{\frac{\alpha+1}{2}}\left[\log\left(\e^{-1}\right)\right]^l+\e c_4[a_{k_0}^{\alpha-1}+
a_{k(\e)}^{\alpha-1}] \\
&\leq c_6\e^{2-\left(2-\frac{\alpha+1}{2}\right)}\left[\log\left(\e^{-1}\right)\right]^l ,
\end{align*}
for all $\varepsilon \in(0,\varepsilon_0)$.

Finally, we apply Theorem \ref{mersat_mod}, where $s=2-\frac{\alpha+1}{2}$.
\end{proof}

The last part of this section is devoted to proving the modified variant of Theorem \ref{tm_ffam}. More precisely, we will prove the modified variant of the original result, \cite[Theorem 5]{zuzu}. To prove this variant, Theorem \ref{ffa} below, we proceed our presentation as in \cite{zuzu}, by first stating and proving where necessary, some auxiliary definitions and results.

We define a spiral in the plane as the graph $\C$ of a function
$r=f(\f)$, $\f\ge\f_1$, in polar coordinates,  
where 
\bgeqn\label{cond}
\left\{
\begin{array}{ll}
\mbox{$f\:[\f_1,\ty)\to(0,\ty)$ is such that  $f(\f)\to0$ as $\f\to\ty$,}&\\
\mbox{$f$ is {\it radially decreasing} (ie, for any fixed $\f\ge\f_1$}&\\
\mbox{the function $\eN\ni k \mapsto f(\f+2k\pi)$
is decreasing)}&\\
\end{array}
\right.
\endeqn

Let $\C$ be a spiral defined by $r=f(\f)$, $\f\ge\f_1$.
We denote a subset of the spiral $\C$ corresponding to angles in the interval $(\f_0,\f_2)$ by 
$\C(\f_0,\f_2)$, more precisely, 
\bgeq
\C(\f_0,\f_2):=\{(r,\f)\in\C\:\f\in(\f_0,\f_2)\}.
\endeq

Let $A$ be a bounded set in $\eR^N$, and let the radial distance function $d_{rad}(x,A)$, be defined as the Euclidean distance from $x$ to the set $A\cap\{tx\:t\ge0\}$, provided the intersection is nonempty, and $\ty$ otherwise. Now the radial $\e$-neighbourhood around $A$ is defined as the set $A_{\e,rad}:=\{y\in\eR^N\:d_{rad}(y,A)<\e\}$.

Using radial $\e$-neighbourhood we define \emph{radial s-dimensional lower and upper Minkowski content} of set $A$, analogously as in Section \ref{subsec_box_dim}, denoted by $\M_*^s(A,rad)$ and $\M^{*s}(A,rad)$, respectively. Also, analogously we define \emph{radial lower and radial upper box dimension} of $A$, denoted by $\underline\dim_B(A,rad)$ and $\ov\dim_B(A,rad)$, respectively. If both quantities coincide, the common value is denoted by $\dim_B(A,rad)$,
and we call it {\it radial box dimension} of $A$. For a general definition of directional box dimensions in $\eR^2$, see Tricot \cite[pp.\ 248--249]{tricot}.
Since $A_{\e,rad}\stq A_\e$, it is clear that 
\bgeq\label{dimrad}
\underline\dim_B(A,rad)\le\underline\dim_BA,\q
\ov\dim_B(A,rad)\le\ov\dim_BA.
\endeq

We define (radial) $\e$-{\it nucleus} of the spiral $\C$ as the 
radial $\e$-neighbourhood around  $\C(\f_2(\e),\ty)\st\C$, that is,
\bgeq\label{N}
N(\C,\e):=\C(\f_2(\e),\ty)_{\e,rad},
\endeq
where by $\f_2(\e)$
we denote the smallest angle such that for all $\psi\ge\f_2(\e)$  we have  $f(\psi)-f(\psi+2\pi)\le2\e$, more precisely,
\bgeq\label{fi2}
\f_2(\e):=\inf\{\f\ge\f_1\:\forall\psi\ge\f,\,\,f(\psi)-f(\psi+2\pi)\le2\e\}.
\endeq

The set $T(\C,\e)$ obtained as the radial $\e$-neighbourhood around the arc $\C(\f_1,\f_2(\e))$, that is,
\bgeq\label{T}
T(\C,\e):=\C(\f_1,\f_2(\e))_{\e,rad},
\endeq 
is called (radial) $\e$-{\it tail} of the spiral $\C$.
The notions of nucleus and tail of a spiral are introduced by Tricot \cite[pp.\ 121, 122]{tricot}.

We consider lower nucleus and lower tail  $s$-dimensional Minkowski contents of $\C$ defined by
\bgeq\label{nt}
\M_*^s(\C,n):=\liminf_{\e\to0}\frac{|N(\C,\e)|}{\e^{2-s}},\q
\M_*^s(\C,t):=\liminf_{\e\to0}\frac{|T(\C,\e)|}{\e^{2-s}}
\endeq
respectively, for $s\ge0$. Analogously for the upper nucleus and upper tail Minkowski contents. 
It is clear that
\bgeq\label{gnc}
\M^{*s}(\C,rad)\le\M^{*s}(\C,n)+\M^{*s}(\C,t).
\endeq
Indeed, we can express radial $\e$-neighbourhood around $\C$ as
 $\C_{\e,rad}=N(\C,\e)\cup T(\C,\e)\cup S(\e)$, where $S(\e):=\{(r,\f)\in\C_{\e,rad}\:\f=\f_2(\e)\}$
 is of the $2$-dimensional Lebesgue measure zero. Hence
\[
\M^{*s}(\C,rad)\le\limsup_{\e\to0}\frac{|N(\C,\e)|}{\e^{2-s}}+\limsup_{\e\to0}\frac{|T(\C,\e)|}
{\e^{2-s}} .
\]

First we prove the modified variant of \cite[Theorem 1]{zuzu}.

\begin{theorem}\label{F}
Let $f\:[\f_1,\ty)\to(0,\ty)$, where $\f_1>e$, be a measurable, radially decreasing function, see
(\ref{cond}). Let $\a\in(0,1)$ and $l\in\eN$ such that 
for some positive numbers $\underline m$ and $\ov m$ we have 
\bgeq\label{m}
\underline m\,\f^{-\a}\left[\log\f\right]^l\le f(\f)\le\ov m\,\f^{-\a}\left[\log\f\right]^l
\endeq
for all $\f\ge\f_1>0$. Assume that there exist positive constants $\underline a$ and $\ov a$ such that for all $\f\ge\f_1$,
\bgeq\label{a}
\underline a\,\f^{-\a-1}\left[\log\f\right]^l\le f(\f)-f(\f+2\pi)\le \overline a\,\f^{-\a-1}\left[\log\f\right]^l.
\endeq

Let
$\C$ be the graph of $r=f(\f)$ in polar coordinates. Then
\bgeqn
&d:=\dim_B(\C,rad)=\frac2{1+\a},&\\ 
&\M^{*d}(\C,rad)=+\infty.&
\endeqn
\end{theorem}

\begin{proof} 
We first obtain the upper bound of the area of the $\e$-nucleus of $\C$. Note that
inequality $f(\f)-f(\f+2\pi)> 2\e$
is satisfied when $\underline a\,\f^{-\a-1}\left[\log\f\right]^l > \underline a\,\f^{-\a-1} > 2\e$, that is, for 
$\f<\underline \f_2(\e)$,
where $\underline\f_2(\e):=\left(\frac{2\e}{\underline a}\right)^{-1/(1+\a)}$. 
From the definition of $\f_2(\e)$, see (\ref{fi2}), we have
\bgeq\label{uf2}
\f_2(\e)\ge\underline\f_2(\e),
\endeq
therefore,
\bgeq
|N(\C,\e)|\le\pi(\sup_{[\underline\f_2(\e),\underline\f_2(\e)+2\pi]} f+\e)^2\le
\pi\left(\ov m\,\underline\f_2(\e)^{-\a}\left[\log\underline\f_2(\e)\right]^l+\e\right)^2.
\endeq 
We see that
\bgeq
|N(\C,\e)|\le\ov c_1\cdot\e^{2\a/(1+\a)}\left[\log\e\right]^{2l},
\endeq
where $\ov c_1>0$.

Now we estimate the area of the $\e$-tail of $\C$ from above.
The inequality $f(\f)-f(\f+2\pi)< 2\e$ is satisfied when $\ov a\,\f^{-\a-1}\left[\log\f\right]^l< 2\e$. Hence, $f(\f)-f(\f+2\pi)< 2\e$ is satisfied for $\f>\ov
\f_2(\e)$,
where 
\bgeq\label{ovf2}
\ov\f_2(\e):=\left(\frac{2\e}{\ov a}\right)^{-1/(1+\a-\d)},
\endeq
for $\d=\d(\e):=\inf_{\d>0}\{\d : [\log\f]^l<\f^{\d}, \forall \f\geq\f_2(\e)\}$. Notice that $\ov a\,\f^{-\a-1}\left[\log\f\right]^l\leq\ov a\,\f^{-(\a-\d)-1}$, for all $\f\geq\f_2(\e)$. 
Therefore $\f_2(\e)\le\ov\f_2(\e)$, and from this
we have that
\bgeqn
|T(\C,\e)|&\le&2\int_{\f_1}^{\ov\f_2(\e)}[(f(\f)+\e)^2-(f(\f)-\e)^2]\,d\f \nonumber\\
&=&2\e\int_{\f_1}^{\ov\f_2(\e)}f(\f)\,d\f\le 2\e\ov m\int_{\f_1}^{\ov\f_2(\e)}\f^{-\a}\left[\log\f\right]^ld\f \nonumber\\
&\le& 2\e\ov m\int_{\f_1}^{\ov\f_2(\e)}\f^{-(\a-\d)}d\f \nonumber\\
&=&\frac{2\e\ov m}{1-\a}(\ov\f_2(\e)^{1-(\a-\d)}-\f_1^{1-(\a-\d)})\le\ov c_2\cdot\e^{2(\a-\d)/(1+\a-\d)}, \nonumber
\endeqn
where
$\ov c_2>0$.
Notice that $\d\to 0$ as $\e\to 0$.

Defining $d:=2/(1+\a)$ we have that, see (\ref{gnc}),
\bgeq
\M^{*d}(\C,rad)\le \M^{*d}(\C,n)+\M^{*d}(\C,t)=+\infty+\infty=+\infty .
\endeq

For every $d'>d$ it holds that $\M^{*d'}(\C,n)=0$, and we can take $\e>0$ sufficiently small, such that $\d>0$ be sufficiently small, so that $\M^{*d'}(\C,t)=0$. Hence, we conclude that $\dim_B(\C,rad)\leq d$.

To obtain a lower bound of the area of $\e$-nucleus of $\C$, 
we show that
\bgeq\label{NB}
N(\C,\e)\supset B_r(0),\q r:=\inf_{\f\in[\ov\f_2(\e),\ov\f_2(\e)+2\pi]} f(\f) ,
\endeq
analogously as in the proof of \cite[Theorem 1]{zuzu}.
Using (\ref{NB}) and (\ref{m}) we obtain
\bgeq\label{NN}
|N(\C,\e)|\ge \pi r^2\ge\pi\left(\underline m(\ov\f_2(\e)+2\pi)^{-\a}\left[\log(\ov\f_2(\e))\right]^l\right)^{2},
\endeq
hence,
\bgeq
|N(\C,\e)|\ge\underline c_1\cdot\e^{2\a/(1+\a-\d)}\left[\log\e\right]^{2l},
\endeq
where $\underline c_1>0$.

Similarly as above we obtain that
\bgeq
|T(\C,\e)|\ge2\e\int_{\f_1}^{\underline\f_2(\e)}f(\f)\,d\f
\ge \underline c_2\cdot\e^{2\a/(1+\a)},
\endeq
where $\underline c_2>0$, provided $\e$ is sufficiently small.

As $\underline\f_2(\e)\leq\overline\f_2(\e)$, we conclude that
\bgeq
\M_*^{d}(\C,rad)\ge\liminf_{\e\to0}\frac{\underline c_1\cdot\e^{2\a/(1+\a-\d)}\left[\log\e\right]^{2l} +\underline c_2\cdot\e^{2\a/(1+\a)}}{\e^{2-d}} = \underline c_2 .
\endeq
For every $d'<d$, we can take $\e>0$ sufficiently small, such that $\d>0$ be sufficiently small, so it holds that $\M_*^{d'}(\C,rad)=+\infty$. Hence, we conclude that $\dim_B(\C,rad)\geq d$.
\end{proof}

The following theorem is a marginally modified variant of \cite[Theorem 4]{zuzu}. The only difference is in adding the log term $\left[\log\e\right]^l$.

\begin{theorem}\label{FI} 
Let $\C$ be a spiral of focus type defined by $r=f(\f)$, $f\:[\f_1,\ty)\to(0,\ty)$,
such that $f(\f)$ is decreasing, and $f(\f)\to0$.
Let there exist $\e_0$ such that the functional inequality
\bgeq\label{fi}
f\left(
\f+\frac{\e}{f(\f)}+\frac{\e}{f(\f)-\e}
\right)
>
f(\f)-\e.
\endeq
holds for all $\e\in(0,\e_0)$ and $\f\in(\f_1,\f_2(\e))$, where $\f_2(\e)$ is defined by (\ref{fi2}).
Assume also that there exist positive constants $\underline C$, $\ov C$ and $q<1$
such that $\underline C\,\e^{q}\le f(\f_2(\e))\le\ov C\,\e^{1-\frac d2}\left[\log\e\right]^l$, where $d:=\ov\dim_B(\C,rad)$ and $l\in\eN$.
Then
\bgeq
\ov\dim_B\C=\ov\dim_B(\C,rad).
\endeq
\end{theorem}

We omit the proof of Theorem \ref{FI}, as it is almost completely analogous to the proof of \cite[Theorem 4]{zuzu}. Only one small difference occurs in the treatment of the log term in the condition $f(\f_2(\e))\le\ov C\,\e^{1-\frac d2}\left[\log\e\right]^l$.

The following excision property of Minkowski contents will enable us to handle the condition for $\f_1$ to be sufficiently large in Theorem~\ref{ffa}. We completely omit the proof, as it is already proved in \cite{zuzu}.

\begin{lemma}\label{excision} {\rm(Excision property for simple smooth curves)}
Let $\C$ be a simple smooth curve in $\eR^2$, that is, $\C$ is the graph of continuously
differentiable injection $h\:[\f_1,\ty)\to\eR^2$. Assume that $\underline\dim_B\C>1$. Let 
$\ov\f_1>\f_1$ be given and $\C_1:=h(\ov\f_1,\ty)$.  Then 
\bgeqn
&\underline d:=\underline\dim_B\C_1=\underline\dim_B\C,\q\ov d:=\ov\dim_B\C_1=\ov\dim_B\C,&\\
&\M_*^{\underline d}(\C_1)=
\M_*^{\underline d}(\C),\q
\M^{*\ov d}(\C_1)=\M^{*\ov d}(\C).&
\endeqn
Analogous claim holds for radial box dimensions and radial Minkowski contents:
if $\underline\dim_B(\C,rad)>1$, then 
\bgeqn
&\underline \d:=\underline\dim_B(\C_1,rad)=\underline\dim_B(\C,rad),\q\ov \d:=\ov\dim_B(\C_1,rad)=
\ov\dim_B(\C,rad),&\nonumber\\
&\M_*^{\underline \d}(\C_1,rad)=
\M_*^{\underline \d}(\C,rad),\q
\M^{*\ov \d}(\C_1,rad)=\M^{*\ov \d}(\C,rad).&\nonumber
\endeqn
In particular, the conclusions hold for smooth spirals $r=f(\f)$, where $f(\f)$ is a decreasing function tending to $0$ as $\f\to\ty$.
\end{lemma}

Finally, here is the modified variant of Theorem \ref{tm_ffam}.

\begin{theorem}\label{ffa} 
Assume in addition to the assumptions of Theorem \ref{F} that the function $f$
is decreasing, of class $C^2$,  and
there exist positive constants $M_1$ and $M_2$ such that for all $\f\ge\f_1$,
\bgeq\label{Mi}
M_1\f^{-\a-1}[\log\f]^l\le|f'(\f)|\le M_2\f^{-\a-1}[\log\f]^l.
\endeq
Then
\bgeq\label{drad}
\dim_B\C=\dim_B(\C,rad)=d,
\endeq
and
\bgeq\label{Mrad}
\M^{*d}(\C)=\M^{*d}(\C,rad)=+\infty,
\endeq
where $d:=\frac2{1+\a}$.
\end{theorem}

\begin{proof}
(a) From the excision result, see Lemma~\ref{excision}, we can assume without loss of generality that $\f_1$ is sufficiently large, which we need below.
We first check that condition (\ref{fi}) of Theorem~\ref{FI} is fulfilled. By the Lagrange mean value theorem
 for all $\f\in(\f_1,\f_2(\e))$, where $\f_2(\e)$ is defined in (\ref{fi2}), we have that
\begin{align*}
D &:= f(\f)-f\left(
\f+\frac{\e}{f(\f)}+\frac{\e}{f(\f)-\e}
\right) \\
&\le |f'(\f)|\left(
\frac\e{f(\f)}+\frac{\e}{f(\f)-\e}
\right) \\
&\le M_2\f^{-\a-1}[\log\f]^l\left(
\frac\e{\underline m\f^{-\a}[\log\f]^l}+\frac{\e}{\underline m\f^{-\a}[\log\f]^l-\e}
\right) \\
&= \e\cdot M_2\frac{\f^{-1}}{[\log\f]^l}\left(
\frac1{\underline m}+\frac1{\underline m-\e\cdot\frac{\f^{\a}}{[\log\f]^l}}
\right).
\end{align*}
Since $\f_2(\e)\le\ov c\cdot\e^{-1/(1+\a-\d)}$, see the proof of Theorem \ref{F}, 
we have
\[
\e\cdot\frac{\f^\a}{[\log\f]^l}\le\e\cdot\f^\a\le\e\cdot\f_2(\e)^\a
\le\ov c^\a\e^{(1-\d)/{(1+\a-\d)}}\le\frac 12\underline m
\]
for all $\e\in(0,\e_0)$, provided
$\e_0$ is sufficiently small. Therefore,
\bgeq
D\le\e\cdot\frac{3M_2}{\underline m\,\f_1[\log\f_1]^l}<\e,
\endeq
where we assume that $\f_1$ is sufficiently large: $\f_1>3M_2/\underline m$.

The second condition in Theorem~\ref{FI} is also fulfilled. Indeed, since 
$\underline c\cdot\e^{-1/(1+\a)}\le\f_2(\e)\le\ov c\cdot\e^{-1/(1+\a-\ov\d)}$, where $\ov\d:=\sup_{\e\in(0,\e_0)} \d(\e)$ and $\d(\e)$ being defined as in the proof of Theorem \ref{F}, we conclude that
\begin{align*}
\underline m\ov c^{-\a}\e^{\a/(1+\a-\ov\d)} &\le \underline m\ov c^{-\a}\e^{\a/(1+\a-\ov\d)}\left[\log\left(\ov c \cdot \e^{-1/(1+\a-\ov\d)}\right)\right]^l \\
&\le f(\f_2(\e))\le\ov m\underline c^{-\a}\e^{\a/(1+\a)},
\end{align*}
that is, $\underline C\e^q\le f(\f_2(\e))\le\ov C\e^{1-\frac d2}$, where $q:=\a/(1+\a-\ov\d)$.
Therefore, by Theorem~\ref{FI} we have that $\ov\dim_B\C=\ov\dim_B(\C,rad)$.
Now from this, using (\ref{dimrad}), Theorems~\ref{F} and~\ref{FI}, we obtain 
$$
\frac2{1+\a}=\underline\dim_B(\C,rad)\le
\underline\dim_B\C\le\ov\dim_B\C=\ov\dim_B(\C,rad)=\frac2{1+\a}.
$$ 
This proves (\ref{drad}).

(b)
To prove (\ref{Mrad}) it suffices to check that for all $\e\in(0,\e_0)$,
\bgeq\label{c}
|\C(\f_1,\ty)_{\e,rad}|-O(\e^2)\le|\C(\f_1,\ty)_\e|.
\endeq
Indeed, since $\C(\f_1,\ty)_{\e,rad}\stq\C(\f_1,\ty)_\e\cup A(\e)$,
where $A(\e)$ is the area of the part of $\C(\f_1,\ty)_\e$ corresponding to $\f<\f_1$. This area is clearly of order $O(\e^2)$ since it is contained in the disk $B_{\e}(T_1)$, 
where $T_1$  is the point on $\C$ corresponding to $\f_1$.

From (\ref{c}) we have $\M^{*s}(\C,rad)\leq\M^{*s}(\C)$, for all $s\geq 0$. From Theorem \ref{F} it follows $\M^{*d}(\C,rad)=+\infty$, hence $\M^{*d}(\C)=+\infty$.
\end{proof}

\paragraph{Acknowledgments.}
This research was supported by: Croatian Science Foundation (HRZZ) under the project IP-2014-09-2285, French ANR project STAAVF 11-BS01-009, French-Croatian bilateral Cogito project \emph{Classification de points fixes et de singularit\'es }\`a\emph{ l'aide d'epsilon-voisinages d'orbites et de courbes}, and University of Zagreb research support for 2015 and 2016.

\bibliographystyle{plain}
\bibliography{bibliografija}

\end{document}